\documentclass[submission]{FPSAC2018}
\usepackage[utf8]{inputenc}
\usepackage{paralist}

\newtheorem{theorem}{Theorem}[section]
\newtheorem{lemma}[theorem]{Lemma}
\newtheorem{corollary}[theorem]{Corollary}
\newtheorem{conjecture}[theorem]{Conjecture}
\newtheorem{proposition}[theorem]{Proposition}
\theoremstyle{definition} %
\newtheorem{definition}[theorem]{Definition}
\newtheorem{remark}[theorem]{Remark}
\newtheorem{example}[theorem]{Example}
\newtheorem{examples}[theorem]{Examples}
\newtheorem{question}[theorem]{Question}

\usepackage{mathrsfs}

\usepackage{ifthen}
\newboolean{draft}
\setboolean{draft}{false}
\newboolean{longversion}
\setboolean{longversion}{false}
\newboolean{inprogress}
\setboolean{inprogress}{false}

\ifdraft
\usepackage{soul}
\usepackage[colorinlistoftodos]{todonotes}
\presetkeys{todonotes}{size=\small,color=green!30,inline,color=red!30}{}
\else
\newcommand{\todo}[2][]{}
\newcommand{\texthl}[1]{}
\fi
\newcommand{\nicolas}[1]{}
\newcommand{\justine}[1]{}

\ifinprogress
\newcommand{\tentative}[1]{}
\fi

\newcommand{\profile}[1]{\varphi_{#1}}
\newcommand{\profileseries}[1]{\mathcal{H}_{#1}}

\newcommand{\QQ}{\mathbb{Q}}
\newcommand{\NN}{\mathbb{N}}
\newcommand{\ZZ}{\mathbb{Z}}
\newcommand{\KK}{\QQ}

\newcommand{\restrict}[2]{{#1}_{|#2}}
\newcommand{\reynolds}[2][]{R_{#1}^{#2}}
\newcommand{\age}[1]{\mathcal{A}(#1)}
\newcommand{\orbitalgebra}[2][\KK]{#1\age{#2}}
\newcommand{\wreath}{\wr}
\newcommand{\sg}{\mathfrak{S}}
\newcommand{\sym}{\operatorname{Sym}}

\newcommand{\Aut}{\operatorname{Aut}}
\newcommand{\suchthat}{\mid}

\newcommand{\fix}{\operatorname{Fix}}
\newcommand{\stab}{\operatorname{Stab}}
\newcommand{\fisubgroup}{K}

\title[]{The orbit algebra of an oligomorphic permutation group with
  polynomial profile is Cohen-Macaulay}

\author{Justine Falque \and Nicolas M. Thiéry}

\address{Univ Paris-Sud, Laboratoire de Recherche en Informatique,
Orsay; %
CNRS, Orsay, F-91405, France}

\received{\today}

\abstract{Let $G$ be a group of permutations of a denumerable set
  $E$. The \textbf{profile} of $G$ is the function $\profile{G}$ which
  counts, for each $n$, the (possibly infinite) number $\profile{G}(n)$ of orbits of $G$
  acting on the $n$-subsets of $E$.
  Counting functions arising this way, and their associated generating
  series, form a rich yet apparently strongly constrained class. In
  particular, Cameron conjectured in the late
  seventies that, whenever $\profile{G}(n)$ is bounded by a
  polynomial, it is asymptotically equivalent to a polynomial. In
  1985, Macpherson further asked if the \textbf{orbit
    algebra} of $G$ -- a graded commutative algebra invented by
  Cameron and whose Hilbert function is $\profile{G}$ -- is finitely
  generated.

  In this paper we announce a proof of a stronger statement: the orbit
  algebra is Cohen-Macaulay. The generating series of the profile is a
  rational fraction whose numerator has positive coefficients and
  denominator admits a combinatorial description.

  The proof uses classical techniques from group actions, commutative algebra, and invariant theory; it steps towards
  a classification of ages of permutation groups with profile bounded
  by a polynomial.}


\usepackage[backend=bibtex]{biblatex}
\begin{document}

\maketitle
\ifdraft
\clearpage
\listoftodos
\clearpage
\fi

\section{Introduction}

Counting objects under a group action is a classical endeavor in
algebraic combinatorics. If $G$ is a permutation group acting on a
finite set $E$, Burnside's lemma provides a formula for the number of
orbits, while Pólya theory refines this formula to compute, for
example, the \textbf{profile} of $G$, that is the function which counts,
for each $n$, the number $\profile{G}(n)$ of orbits of $G$ acting on
subsets of size $n$ of $E$.

In the seventies, Cameron initiated the study of the profile when $G$
is instead a permutation group of an infinite set $E$. Of course the
question makes sense mostly if $\profile{G}(n)$ is finite for all $n$;
in that case, the group is called \textbf{oligomorphic}, and the
infinite sequence $\profile{G}=(\profile{G}(n))_n$
an \textbf{orbital profile}. This
setting includes, for example, counting integer partitions (with
optional length and height restrictions) or graphs up to an
isomorphism, and has become a whole research
subject~\cite{Cameron.1990,Cameron.OPG.2009}. One central topic is the
description of general properties of orbital profiles.
It was soon observed that the potential growth rates exhibited jumps.
For example, the profile either grows at least as fast as the
partition function, or is bounded by a polynomial~\cite[Theorem~1.2]{Macpherson.1985}.
In the latter case, it was conjectured to be asymptotically polynomial:
\begin{conjecture}[Cameron~\cite{Cameron.1990}]
  \label{conjecture.cameron}
  Let $G$ be an oligomorphic permutation group whose profile is
  bounded by a polynomial. Then $\profile{G}(n)\sim a n^k$ for some
  $a>0$ and $k\in \NN$.
\end{conjecture}

As a tool in this study, Cameron introduced early on the \textbf{orbit
  algebra} $\orbitalgebra{G}$ of $G$, a graded connected commutative
algebra whose Hilbert function coincides with $\profile{G}$.
Macpherson then asked the following

\begin{question}[Macpherson~\cite{Macpherson.1985} p.~286]
  Let $G$ be an oligomorphic permutation group whose profile is
  bounded by a polynomial. Is $\orbitalgebra{G}$ finitely generated?
\end{question}
The point is that, by standard commutative algebra, whenever
$\orbitalgebra{G}$ is finitely generated, its Hilbert function is
asymptotically polynomial, as conjectured by Cameron. It is in fact \emph{eventually
  a quasi polynomial}. Equivalently, the
generating series of the profile
$\profileseries{G}=\sum_{n\in \NN}\profile{G}(n)z^n$ is a rational
fraction of the form
\begin{displaymath}
  \profileseries{G} = \frac {P(z)}{\prod_{i\in I}(1-z^{d_i})}\,,
\end{displaymath}
where $P(z)$ is a polynomial in $\ZZ[z]$ and the $d_i$'s are the degrees of the
generators.

In this paper, we report on a proof of Cameron's conjecture by
answering positively to Macpherson's question, and even to a stronger
question: is $\orbitalgebra{G}$ Cohen-Macaulay?
\begin{theorem}[Main Theorem]
  \label{theorem.cohen-macaulay}
  Let $G$ be a permutation group whose profile is bounded by a
  polynomial. Then $\orbitalgebra{G}$ is Cohen-Macaulay over a free
  subalgebra with generators of degrees $(d_i)_{i\in I}$ prescribed by
  the \textbf{block structure} of $G$.
\end{theorem}
\begin{corollary}
  \label{corollary.hilbert_series}
  The generating series of the profile is of the (irreducible) form
  \begin{displaymath}
    \profileseries{G} = \frac {P(z)}{\prod_{i\in I}(1-z^{d_i})}\,,
  \end{displaymath}
  for some polynomial $P$ in $\NN[z]$; therefore
  $\profile{G}\sim an^{|I|-1}$ for some $a>0$.
\end{corollary}

Investigating the Cohen-Macaulay property was inspired by the
important special case of invariant rings of finite permutation
groups. At this stage, we presume that this theorem can be obtained as
a consequence of some classification result for ages of permutation
groups with profile bounded by a polynomial. In addition, the orbit
algebra would always be isomorphic to the invariant ring of some
finite permutation group.

This research is part of a larger program initiated in the seventies:
the study of the \emph{profile of relational
  structures}~\cite{Fraisse.TR.2000,Pouzet.2006.SurveyProfile} and in
general of the behavior of counting functions for hereditary classes of
finite structures, like undirected graphs, posets, tournaments,
ordered graphs, or permutations; see~\cite{Klazar.2008.Overview,Bollobas.1998.HereditaryPropertiesOfGraphs} for surveys.
For example, the analogue of Cameron's conjecture is proved
in~\cite{Balogh_Bollobas_Saks_Sos.2009} for undirected graphs and
in~\cite{Kaiser_Klazar.2003} for permutations.

In~\cite{Pouzet.2008.IntegralDomain}, the orbit algebra is proved to
be integral (under a natural restriction).
Cameron extends
in~\cite{cameron.1997} the definition of the orbit algebra to the
general context of relational structures. The analogue of
Theorem~\ref{theorem.cohen-macaulay} holds when the profile is
bounded (see~\cite[Theorem~26]{Pouzet.2006.SurveyProfile} and
\cite[Theorem~1.5]{Pouzet_Thiery.AgeAlgebra1}); it can fail as soon as
the profile grows faster.

In a context which is roughly the generalization of transitive groups
with a finite number of infinite blocks, the analogue of Cameron's
conjecture holds~\cite[Theorem~1.7]{Pouzet_Thiery.AgeAlgebra1} and the
finite generation admits a combinatorial
characterization~\cite{Pouzet_Thiery.AgeAlgebra2}.

This paper is structured as follows. In
Section~\ref{section.preliminaries}, we review the basic definitions
of orbit algebras and provide classical examples and operations. The
central notion is that of block systems; we explain that each block system
provides a lower bound on the growth of the profile and state the
existence of a \emph{canonical block system} $B(G)$ meant to maximize this
lower bound. This block system splits into orbits of blocks; each such
orbit consists either of infinitely many finite blocks or finitely
many infinite blocks; there can be in addition some finite orbits of finite
blocks; they form the \emph{kernel} and are mostly harmless. The
approach in the sequel is to treat each orbit of blocks separately, and then
recombine the results.

In Section~\ref{section.subgroups_of_finite_index}, we show that the
finite generation property can be lifted from normal subgroups of
finite index. This proves our Main Theorem when $B(G)$ consists of a
single orbit of infinite blocks. This further enables to assume
without loss of generality that the kernel is empty. In
Section~\ref{section.finite_blocks}, we classify, up to taking
subgroups of finite index, ages when $B(G)$ consists of a single
infinite orbit of finite blocks. Finally, in
Section~\ref{section.decoupage_recollage}, we combine the previous
results to prove our Main Theorem in the general case.

Full proofs and additional examples and figures will be published in a
long version of this extended abstract.

We would like to thank Maurice Pouzet for suggesting to work on this
conjecture, and Peter Cameron and Maurice Pouzet for enlightening
discussions.

\section{Preliminaries}
\label{section.preliminaries}

\subsection{The age, profile and orbit algebra of a permutation group}
\label{subsection.profile.age}

Let $G$ be a permutation group, that is a group of permutations of some
set $E$. Unless stated otherwise, $E$ is denumerable and $G$ is infinite.
The action of $G$ on the elements of $E$ induces an action on finite
subsets. The \textbf{age} of $G$ is the set $\age{G}$ of its
orbits of finite subsets. Within an orbit, all subsets share the
same cardinality, which is called the \textbf{degree} of the orbit. This
gives a grading of the age according to the degree of the orbits:
$\age{G} = \sqcup_{n\in \NN} \age{G}_n$. The \textbf{profile} of $G$ is
the function $\profile{G}: n\mapsto |\age{G}_n|$. In general, the
profile may take infinite values; the group is called
\textbf{oligomorphic} if it does not. 

We call the growth rate of a profile bounded by a polynomial the smallest number
$r$ satisfying $\phi (n) = O(n^r)$
(for instance, the growth rate of $n^2 + n$ is $2$). %
By extension, we speak of the \textbf{growth rate of a permutation group}
(which is that of its profile).

\begin{definition}
  We say that a permutation group is \textbf{$P$-oligomorphic} if its
  profile is bounded by a polynomial.
\end{definition}

\begin{examples}
  Let $G$ be the infinite symmetric group $\mathfrak{S}_{\infty}$. For
  each $n$ there is a single orbit containing all subsets of size $n$,
  hence $\profile{G}(n)=1$. We say that $G$ \emph{has profile $1$}.

  Now take $E=E_1 \sqcup E_2$, where $E_1$ and $E_2$ are two copies of
  $\NN$. Let $G$ be the group acting on $E$ by permuting the elements
  independently within $E_1$ and $E_2$ and by exchanging $E_1$ and
  $E_2$: $G$ is the \textbf{wreath product} $\sg_\infty \wreath \sg_2$. In that
  case, the orbits of subsets of cardinality $n$ are in bijection with
  integer partitions of $n$ with at most two parts.
\end{examples}

The \textbf{orbit algebra} of $G$, as defined by Cameron, is the 
graded connected vector
space $\orbitalgebra{G}$ of formal finite linear combinations of
elements of $\age{G}$, endowed with a commutative product as
follows: embed $\orbitalgebra{G}$ in the vector space of formal linear
combinations of finite subsets of $E$ by mapping each orbit to the sum of
its elements; endow the latter with the disjoint product that maps two
finite subsets to their union if they are disjoint and to $0$
otherwise. Some care is to be taken to check that all is
well defined.

\subsection{Block systems and primitive groups}

A key notion when studying permutation groups is that of \textbf{block
  systems}; they are the discrete analogues of quotient modules in
representation theory. A \textbf{block system} is a partition of $E$
into parts, called \textbf{blocks}, such that each $g\in G$ maps blocks
onto blocks. The partitions $\{E\}$ and $\{\{e\} \suchthat e\in E\}$
are always block systems and are therefore called the trivial block systems.
A permutation group is \textbf{primitive} if it admits no non trivial
block system. By extension, an orbit of elements is \textbf{primitive} if
the restriction of the group to this orbit is primitive.
A block system is \textbf{transitive} if $G$ acts transitively on its
blocks; in this case, all the blocks are conjugated and thus share the
same cardinality.

The collection of block systems of a permutation group forms a lattice
with respect to the refinement order, with the two trivial block
systems as top and bottom respectively.

The following two theorems will be central in
our study.
\begin{theorem}[Macpherson~\cite{Macpherson.1985}]
  \label{primitive.hightly.homogeneous}
  The profile of an oligomorphic primitive permutation group is either
  $1$ or grows at least as fast as the partition function.
\end{theorem}
In our study of Macpherson's question, all groups have profile
bounded by a polynomial, and therefore primitive groups always have
profile $1$.

\begin{theorem}[Cameron~\cite{Cameron.1990}]
  \label{profile.one.classification}
  There are only five complete permutation groups with profile $1$:
  \begin{compactenum}
  \item The automorphism group $\Aut(\QQ)$ of the rational chain
    (order-preserving bijections on $\QQ$);
  \item $\mathrm{Rev}(\QQ)$ (order-preserving or reversing bijections
    on $\QQ$)
  \item $\Aut(\QQ / \mathbb{Z})$, preserving the cyclic order (see $\QQ / \mathbb{Z}$ as a circle);
  \item $\mathrm{Rev}(\QQ / \mathbb{Z})$, generated by $\mathrm{Cyc}(\QQ / \mathbb{Z})$ and a reflection;
  \item $\mathfrak{S}_{\infty}$.
  \end{compactenum}
\end{theorem}

The notion of completion refers here to the topology of simple
convergence, described in Section~2.4 of \cite{Cameron.1990}. Thanks
to the following lemma, it plays only minor role for our purposes.
Thus, for the sake of the simplicity of exposition in this extended
abstract, we will most of the time blur the distinction between a
permutation group and its completion.
\begin{lemma} \label{lemma.completion_same_profile_and_orbits}
A permutation group and its completion share the same profile and orbit algebra.
\end{lemma}
In particular, we may now say that
Theorem~\ref{profile.one.classification} essentially classifies
primitive $P$-oligomorphic groups.

\subsection{Operations and examples}

\begin{lemma}[Operations on groups, their ages and orbit algebras]\ 
  \label{lemma.operations}
  \begin{compactenum}
  \item Let $G$ be a permutation group acting on $E$, and $F$ be a
    stable subset of $E$. Then, $\orbitalgebra{\restrict{G}{F}}$ is both a
    subalgebra and a quotient of $\orbitalgebra{G}$.
  \item Let $G$ be a permutation group acting on $E$ and $H$ be a
    subgroup. Then, $\orbitalgebra{G}$ is a subalgebra of
    $\orbitalgebra{H}$.
  \item \label{lemma.operations.direct_product}
    Let $G$ and $H$ be permutation groups acting on $E$ and $F$
    respectively. Take $G\times H$ endowed with its natural action on the disjoint union $E \sqcup F$.
    Then, $\age{G\times H}\simeq \age{G}\times\age{H}$,
    and
    $\orbitalgebra{G\times H} \simeq \orbitalgebra{G} \otimes
    \orbitalgebra{H}$; it follows that
    $\profileseries{G\times H}=\profileseries{G}\profileseries{H}$.
  \item Let $G$ and $H$ be permutation groups acting on $E$ and $F$
    respectively. Intuitively, the \textbf{wreath product} $G\wreath H$ acts on
    $|F|$ copies $(E_f)_{f\in F}$ of $E$, by permuting each
    copy of $E$ independently according to $G$ and permuting the
    copies according to $H$. By construction, the partition
    $(E_f)_{f\in F}$ forms a block system, and $G\wreath H$ is not
    primitive (unless $G$ or $H$ is and $F$ or $E$, respectively, 
    is of size 1).
  \end{compactenum}
\end{lemma}

\begin{examples}[Wreath products]\ 
  \label{example.wreath_products}
  \begin{compactenum}
  \item Let $G$ be the wreath product $\sg_\infty \wreath \sg_k$. The
    profile counts integer partitions with at most $k$ parts.
    The orbit algebra is the algebra of symmetric polynomials over $k$
    variables, that is the free commutative algebra with generators of
    degrees $1,\dots,k$. The generating series of the profile is given
    by $\profileseries{G} = \frac{1}{\prod_{d=1,\ldots,k} (1-z^d)}$.
  \item Let $G'$ be a finite permutation group. Then, the orbit
    algebra of $G=\mathfrak{S}_{\infty} \wreath G'$ is the
    \textbf{invariant ring} $\KK [X]^{G'}$ which consists of the
    polynomials in $\KK[X]=\KK[X_1,\ldots,X_k]$ that are invariant
    under the action of $G'$.
  \item Let $G'$ be a finite permutation group. Then, the orbit
    algebra of $G= G'\wreath \mathfrak{S}_{\infty}$ is the free
    commutative algebra generated by the $G'$-orbits of non trivial
    subsets. The generating series of the profile is given by
    $\profileseries{G} = \frac{1}{\prod_{d} (1-z^d)}$, where $d$ runs
    through the degrees of the $G'$-orbits of non trivial subsets,
    taken with multiplicity.
  \end{compactenum}
\end{examples}

\subsection{The canonical block system}
\label{subsection.block_dimension_and_canonical}

Let $G$ be a $P$-oligomorphic permutation group. In this section we
show how a block system provides a lower bound on the growth of the
profile. Seeking to optimize this lower bound, we establish the
existence of a canonical block system satisfying appropriate
properties. The later sections will show that this block system
minimizes synchronization and provides a tight lower bound.

\begin{lemma}
  \label{lemma.block_degrees}
  Let $G$ be a $P$-oligomorphic permutation group,
  endowed with a transitive block system $B$. Then,
  \begin{compactenum}
  \item Case 1: $B$ has finitely many infinite blocks, as in
    Example~\ref{example.wreath_products} (1) and (2). Then $G$ is a
    subgroup of $\sg_{\infty} \wreath \sg_k$ (where $k$ is the number of
    blocks), and $\orbitalgebra{G}$ contains $\sym_k$ which is a free
    algebra with generators of degrees $(1,\dots,k)$.
  \item Case 2: $B$ has infinitely many finite blocks, as in
    Example~\ref{example.wreath_products} (3). Then, $G$ is a subgroup
    of $\sg_k\wreath \sg_{\infty}$, and $\orbitalgebra{G}$ contains
    the free algebra with generators of degrees $(1,\dots,k)$.
  \end{compactenum}
\end{lemma}

\begin{proof}[Sketch of proof]
  Use Lemma~\ref{lemma.operations} and
  Examples~\ref{example.wreath_products}.
\end{proof}

\begin{remark} \label{algebra_independant_parts}
  Let $G$ be an oligomorphic permutation group, and $E_1,\ldots,E_k$ be
  a partition of $E$ such that each $E_i$ is stable under $G$.
  In our use case, we have a block system $B$, and each $E_i$ is the
  support of one of the orbits of blocks in $B$.

  Then, $G$ is a subgroup of
  $\restrict{G}{E_1}\times\cdots\times\restrict{G}{E_k}$. Therefore,
  by Lemma~\ref{lemma.operations}, $\orbitalgebra{G}$ contains as
  subalgebra
  $\orbitalgebra{\restrict{G}{E_1}}\otimes\cdots\otimes\orbitalgebra{\restrict{G}{E_k}}$.
  In particular, the algebraic dimension of the age algebra is bounded
  below by the sum of the algebraic dimensions for each orbit of blocks.
\end{remark}

Combining Lemma~\ref{lemma.block_degrees} and
Remark~\ref{algebra_independant_parts}, each block system of $G$
provides a lower bound on the algebraic dimension of
$\orbitalgebra{G}$, and therefore on the growth rate of the profile.

This bound is not necessarily tight. Consider indeed the group
$\operatorname{Id}_2 \wreath \sg_\infty$. There are three block
systems: one with two infinite blocks, one with two orbits of blocks
of size $1$, one with one orbit of blocks of size $2$. With the above
considerations, all three block systems give a lower bound of $2$.
However, as we will see in Corollary~\ref{corollary.finite_blocks},
the lower bound for the latter block system can be refined to $3$
which is tight.

This rightfully suggests that better lower bounds are obtained when
maximizing the size of the finite blocks, and then maximizing the
number of infinite blocks.
\begin{theorem}
  \label{theorem.canonical_block_system}
  Let $G$ be a $P$-oligomorphic permutation group. There exists a
  unique block system, denoted $B(G)$ and called the \textbf{canonical
    block system of $G$}, that maximizes the size of its finite
  blocks, and then maximizes the number of infinite blocks.
\end{theorem}
\begin{proof}[Sketch of proof]
  In the lattice of block systems of $G$, consider the lower set $I$
  of the block systems whose blocks are all finite. The join of two
  such block systems turns out to be again in $I$. Furthermore, if $B$
  is coarser than $B'$, then the aforementioned lower bound on the
  algebraic dimension increases strictly; since the growth of the
  profile is polynomial, there can
  be no infinite increasing chain. It follows that $I$ is in fact a
  lattice with a maximal (coarsest) block system. The kernel of $G$,
  if non empty, is one of the blocks of this block system.

  Consider now the collection of the remaining blocks of size $1$ of
  $B$. Following similar arguments as above,
  they can be replaced by a canonical maximal collection of infinite
  blocks to produce $B(G)$.
\end{proof}

\subsection{Subdirect products}
\label{subsection.subdirect_and_synchro}

The actions of a permutation group on two of its orbits need not be
independent; intuitively, there may be partial or full
synchronization, which influences the profile and the orbit algebra. A
classical tool to handle this phenomenon is that of subdirect products.

 \begin{definition}
 Let $G_1$ and $G_2$ be groups. A \textbf{subdirect product} of $G_1$ and 
 $G_2$ is a subgroup of $G_1 \times G_2$ which projects onto each 
factor under the canonical projections.
 \end{definition}

For instance, suppose $G$ is a permutation
group that has two orbits of elements $E_1$ and $E_2$. 
If $G_i$ is the group induced on $E_i$ by $G$, $G$ is a subdirect product
of $G_1$ and $G_2$.

Denote $N_1 = \fix_G (E_2)$ and $N_2 = \fix_G (E_2)$. Then $N_1$ and $N_2$ are normal
subgroups of $G$; and $N_1 \cap N_2 = \{1\}$, so $N_1$ and $N_2$ generate their direct product.
We have (after restriction of $N_i$ when needed): 
$\frac{G_1}{N_1} \simeq \frac{G}{N_1 \times N_2} \simeq \frac{G_2}{N_2}$,
this quotient representing the synchronized part of each group,
and $G = \{ (g_1, g_2) \suchthat g_1 N_1 = g_2 N_2 \}$.

With the classification of the groups of profile 1 in
mind (see Theorem~\ref{profile.one.classification}),
there are very few possible synchronizations
between two primitive infinite orbits of $G$: the possibilities
are indeed linked to the normal subgroups of the two restrictions $G_1$ and $G_2$.
Namely, the synchronization may be total, limited to the reflection in the cases
of $\mathrm{Rev}(\mathbb{Q})$ and $\mathrm{Rev}(\mathbb{Q}/\mathbb{Z})$
(synchronization \emph{of order 2})
or absent.

\begin{lemma}[Reduction 0]
  Synchronizations of order 2 between primitive orbits do not change
  the age of $G$ (this would be false for orbits of tuples).
\end{lemma}

This lemma implies that the synchronizations of order 2 can harmlessly be
ignored in the study of the orbit algebra of a group; so we will
assume from now on that only full synchronizations may exist.

This is also true regarding the infinite orbits of finite blocks 
(instead of just elements),
with the last item of Lemma~\ref{lemma.block_degrees} in mind:
taking two such orbits, either the permutations of their blocks 
fully synchronize (blockwise) or they
do not at all. We derive the following remark.

\begin{remark}
\label{remark.independance_in_canonical_block_system}
By construction, the actions of $G$ on the orbits of blocks of the canonical
 block system $B(G)$ are independant blockwise (potentially ignoring harmless synchronizations of order 2).
\end{remark}

\section{Lifting from subgroups of finite index}
\label{section.subgroups_of_finite_index}

In this section, we study how the orbit algebra of an oligomorphic
permutation group $G$ relates to the orbit algebra of a normal

subgroup $\fisubgroup$ of $G$ of finite index, and derive two important reductions for
Macpherson's question.

\begin{theorem}
  \label{theorem.finite_index.finite_generation}
  Let $G$ be an oligomorphic permutation group and $\fisubgroup$ be a normal
  subgroup of finite index. If the orbit algebra $\orbitalgebra{\fisubgroup}$ of
  $\fisubgroup$ is finitely generated, then so is its subalgebra $\orbitalgebra{G}$.
  If in addition $\orbitalgebra{\fisubgroup}$ is Cohen-Macaulay, then
  $\orbitalgebra{G}$ is Cohen-Macaulay too.
\end{theorem}

This is a close variant of Hilbert's theorem stating that the ring of
invariants of a finite group is finitely generated; the orbit algebra
$\orbitalgebra{\fisubgroup}$ plays the role of the polynomial ring $\KK[X]$,
while the orbit algebra $\orbitalgebra{G}$ plays the role of the
invariant ring $\KK[X]^G$.
The key ingredient in Hilbert's proof is the \emph{Reynolds operator},
a finite averaging operator over the group. In the setting of
orbit algebras, $G$ is not finite; however, we will compensate by using
the \emph{relative Reynolds operator} with respect to $\fisubgroup$, which is a finite
averaging operator over the coset representatives. 
Then we  just
proceed as in Hilbert's proof. The same approach can be used to prove
that $\orbitalgebra{G}$ is Cohen-Macaulay as soon as
$\orbitalgebra{\fisubgroup}$ is.

\begin{corollary}[Reduction 1]
  \label{corollary.reduction.finite_block_system}
  Let $G$ be an oligomorphic group that admits a non trivial finite 
  transitive block system. Let $\fisubgroup$ be the subgroup of
  the elements of $G$ that stabilize each block. Assume that the orbit
  algebra of $\fisubgroup$ is finitely generated. Then so is the orbit algebra
  of $G$.
\end{corollary}

The second application is a reduction of Macpherson's question to
groups that admit no finite orbits of elements.
\begin{definition}
  The \textbf{kernel} of an oligomorphic permutation group $G$ is the
  union $\ker(G)$ of its finite orbits of elements.
\end{definition}
This terminology comes from the broader context of relational
structures: it can be shown that $\ker(G)$ is indeed the kernel of the
associated homogeneous relational structure. It is not to be confused
with the notion of kernel from group theory.
\begin{remark}
  The kernel $\ker(G)$ of an oligomorphic group is finite. Indeed, $G$
  has a finite number $\profile{G}(1)$ of orbits and thus of finite
  orbits; hence their union is finite as well.
\end{remark}

\begin{theorem}[Reduction 2]
  \label{theorem.kernel_reduction}
  Let $G$ be an oligomorphic permutation group with profile bounded by
  a polynomial. Assume that the orbit algebra of any group with the
  same profile growth and no finite orbit is finitely generated. Then,
  the orbit algebra of $G$ is finitely generated as well.
\end{theorem}

\begin{proof}[Sketch of proof]
  Let $\fisubgroup$ be the subgroup of $G$ fixing $\ker(G)$; it is
  normal and of finite index in $G$. Using the two upcoming simple
  lemmas together with
  Lemma~\ref{lemma.operations}~\eqref{lemma.operations.direct_product},
  the groups $\fisubgroup$ and $G$ share the same profile growth.
  Apply Theorem~\ref{theorem.finite_index.finite_generation}.
\end{proof}

\begin{lemma}
  \label{lemma.subgroup.rate}
  Let $G$ be a permutation group and $\fisubgroup$ be a normal subgroup of
  finite index. Then,
  \begin{displaymath}
    \profile{G}(n) \ \leq\ \profile{\fisubgroup}(n) \ \leq\ |G:\fisubgroup| \profile{G}(n)\,.
  \end{displaymath}
\end{lemma}

\begin{lemma} \label{same_kernel}
  Let $G$ be an oligomorphic permutation group and $\fisubgroup$ be a normal
  subgroup of finite index. Then $\ker(\fisubgroup)=\ker(G)$.
\end{lemma}
\begin{proof}
  Let $O$ be a $G$-orbit of elements.  Since $\fisubgroup$ is a normal subgroup, $O$ splits
  into $\fisubgroup$-orbits on which $G$ -- and actually $G/\fisubgroup$ -- acts transitively by
  permutation; there are thus finitely many such $\fisubgroup$-orbits, all of the same
  size. In particular, infinite $G$-orbits split into infinite
  $\fisubgroup$-orbits, and similarly for finite ones.
\end{proof}

 In order to give an idea of the proof of 
Theorem~\ref{theorem.finite_index.finite_generation}, let 
us now turn to the \emph{relative Reynolds operator}
$\reynolds[\fisubgroup]{G}$. It is defined by choosing some representatives
$(g_i)_i$ of the cosets of $\fisubgroup$ in $G$:
\begin{displaymath}
  \reynolds[\fisubgroup]{G} := \frac1{|G:\fisubgroup|} \sum_i g_i\,.
\end{displaymath}

\begin{lemma}
  \label{lemma.relative_reynolds}
  Let $G$ be an oligomorphic permutation group, and $\fisubgroup$ be a normal
  subgroup of finite index. Then, the relative Reynolds operator
  $\reynolds[\fisubgroup]{G}$ defines a projection from $\orbitalgebra{\fisubgroup}$ onto
  $\orbitalgebra{G}$ which does not depend on the choice of the $g_i$'s,
  and is a $\orbitalgebra{G}$-module morphism.
\end{lemma}

\begin{proof}[Sketch of proof of Theorem~\ref{theorem.finite_index.finite_generation}]
  Use the relative Reynolds operator to replay Hilbert's proof of
  finite generation for invariants of finite groups, as well as the
  classical proof of the Cohen-Macaulay property (see
  e.g.~\cite{Stanley.1979}).
\end{proof}

\section{Case of a transitive block system with finite blocks}
\label{section.finite_blocks}
\label{hard_case}

In this section, permutation groups are assumed to be $P$-oligomorphic
and endowed with a non trivial
transitive block system with infinitely many finite blocks, which 
we choose maximal. 
We bring a positive answer to Macpherson's question in this setting.

\begin{lemma}
\label{lemma.action_on_finite_blocks}
Let $G$ be a $P$-oligomorphic permutation group having a non trivial
transitive block system with infinitely many maximal finite blocks;
then $G$ acts on the blocks as $\sg_\infty$.
\end{lemma}

\begin{proof}[Sketch of proof]
Using their maximality, the action of $G$ on the blocks
is classified by Theorem~\ref{profile.one.classification};
argue now that with any of the five actions but $\sg_\infty$ the
profile would not be bounded by any polynomial.
\end{proof}

\begin{definition}
  Let $S_B = S_B^G = \stab_G(B)$ and, for $i\geq 0$,
  $H_i = H_i^G = \fix_{S_B}(B_1, \ldots, B_i)_{|B_{i+1}}$. The
  sequence $H_0 ~ H_1 ~ H_2 \cdots$ is called the \textbf{tower} of
  $G$ with respect to the block system $B$. The groups $H_i$ are
  considered up to a permutation group isomorphism.
\end{definition}

\begin{remark}
  \label{remark.tower_indep_of_order}
  By conjugation, using Lemma~\ref{lemma.action_on_finite_blocks}, 
  the tower does not depend on the ordering of the
  blocks. Furthermore, $H_{i+1}$ is a normal subgroup of $H_i$ for all $i\in \NN$.
  The above definition and this remark also apply
  to a permutation group of a finite set, if it acts on the
  blocks as the full symmetric group.
\end{remark}

\begin{example}
  Let $H$ be a finite permutation group. The tower of
  $H \wr \sg_{\infty}$ (resp. $H \times \sg_{\infty}$) for its natural
  block system is $H ~ H ~ H \cdots$ (resp. $H ~ Id ~ Id \cdots$).
\end{example}

\begin{lemma} \label{lemma.sympa} For all $k \in \mathbb{N}$,
  $\fix_G (B_1, \ldots, B_k)$ acts on the remaining blocks as $\sg_\infty$.
\end{lemma}
\begin{proof}[Sketch of proof]
  As $\fix_G (B_1, \ldots, B_k)$ is a normal subgroup of finite index
  of $\stab_G(B_1,\ldots,B_k)$, it acts on the remaining blocks as a 
  subgroup of finite
  index of $\sg_{\infty}$. Use Theorem~\ref{profile.one.classification}.
\end{proof}

\begin{proposition}
  \label{proposition.tower_ages}
  Two groups with the same tower have isomorphic ages.
\end{proposition}

\begin{proof}[Sketch of proof]
  Start by restricting to the first $k$ blocks, and use
  Lemma~\ref{lemma.sympa} to show that, up to conjugation within each
  block, the restrictions of the two groups have the same age. Use
  again Lemma~\ref{lemma.sympa} to derive that the age of each group
  coincides, up to degree $k$, with that of its restriction. Conclude
  by taking the limit at $k=\infty$.
\end{proof}

\begin{lemma}
  \label{lemma.four_blocks}
  Let $G$ be a finite permutation group endowed with a block system
  composed of
  four blocks on which $G$ acts by $\sg_4$; denote its tower by
  $H_0 ~ H_1 ~ H_2 ~ H_3$. Then, $H_1 = H_2$.
\end{lemma}
\begin{proof}
An element $s$ of $S_B$ is determined by its action on each block,
which we write as a quadruple.
Let $g$ be an element of $H_1$. Then $S_B$ has an element $x$ that may be written $(1, g, h, l)$, with $h$ and $l$ also in $H_1$.
Let $\sigma$ be an element of $G$ that permutes the first two blocks and fixes the
other two (it exists by hypothesis).
By conjugating $x$ with $\sigma$ in $G$, we get an element $y$ in $S_B$
that we may write $(g', 1, h, l)$.
Then we have $x^{-1}y = (g', g^{-1}, 1, 1)$; hence,
using Remark~\ref{remark.tower_indep_of_order},
$g^{-1}$ and therefore $g$ are in $H_2$.
\end{proof}

\begin{corollary}
  \label{corollary.tower}
  Let $G$ be a $P$-oligomorphic permutation group having a non trivial
  transitive block system with infinitely many maximal finite blocks.
  Then, the tower of $G$ has the form $H_0~ H~ H~ H \cdots$, where
  $H_0$ is a finite permutation group and $H$ is a normal subgroup of
  $H_0$. 
\end{corollary}
\begin{proof}[Sketch of proof]
  Restrict to sequences of four consecutive blocks and use
  Lemma~\ref{lemma.four_blocks}.
\end{proof}

\begin{corollary}
\label{corollary.finite_blocks}
  Let $G$ be a $P$-oligomorphic permutation group
  endowed with a non trivial transitive block system with
  infinitely many (maximal) finite blocks. Then, $G$ contains a finite
  index subgroup $\fisubgroup$ whose age coincides with that of
  $H\wreath \sg_\infty$ (where $H$ is as in 
  Corollary~\ref{corollary.tower}).\\
  It follows that its algebraic dimension is given by the number of $H$-orbits 
  (of non trivial subsets).
\end{corollary}

\section{Proof of the main theorem}
\label{section.decoupage_recollage}

\begin{theorem}
  Let $G$ be a $P$-oligomorphic permutation group.
  Then, its orbit algebra $\orbitalgebra{G}$ is finitely generated and Cohen-Macaulay.
\end{theorem}
\begin{proof}[Sketch of proof]

  Consider the canonical block system $B(G)$ introduced in
  subsection~\ref{subsection.block_dimension_and_canonical}. Recall
  that it consists of finitely many infinite blocks, a finite
  number of infinite orbits of finite blocks, and possibly one finite
  stable block.

  We aim to prove the existence of a normal subgroup $\fisubgroup$ of
  finite index of $G$ with a simple form, ensuring that its orbit
  algebra is a finitely generated (almost free) algebra.

  Start with $\fisubgroup=G$.
  If there exists a finite stable block $B_k$ in $B(G)$, replace
  $\fisubgroup$ by the kernel of its action on $B_k$ (i.e. the kernel of
  the natural projection on this action). This ensures that
  $\fisubgroup$ fixes $B_k$.
  Replace further $\fisubgroup$ by the kernel of its action on the set
  of infinite blocks of $B(G)$. This ensures that $\fisubgroup$
  stabilizes each of the infinite blocks. Using that $\fisubgroup$ is of
  finite index and the construction of $B(G)$, 
  these blocks are now primitive orbits. Possibly
  replacing again $\fisubgroup$, we may further assume that (the
  completion of) $\fisubgroup$ acts on each of them as one of the three
  primitive groups of Theorem~\ref{profile.one.classification} that
  admit no subgroup of finite index.
  Now take an orbit of finite blocks. Using Corollary~\ref{corollary.finite_blocks}, 
  and replacing $\fisubgroup$ if needed, we may
  assume that the restriction of $\fisubgroup$ on the support of the orbit has
  the same age as some $H\wreath \sg_\infty$. Repeat for the other
  orbits of finite blocks.

  By construction of $B(G)$ and
  Subsection~\ref{subsection.subdirect_and_synchro}, argue that there
  is now no synchronization between $K$-orbits of blocks. 
  
  Then, $\fisubgroup$ has the same age as some direct product of
  groups of the form $\sg_\infty$, $\Aut(\QQ)$,
  $\Aut(\QQ / \mathbb{Z})$, and $G'\wreath \sg_\infty$ (and possibly a
  finite identity group). From Remark~\ref{algebra_independant_parts},
  $\orbitalgebra{\fisubgroup}$ is a finitely generated free algebra (possibly tensored
  with some finite dimensional diagonal algebra). Using
  Theorem~\ref{theorem.finite_index.finite_generation}, it follows
  that $\orbitalgebra{G}$ is finitely generated
  and Cohen-Macaulay over some free subalgebra $\QQ[\theta_i]$.

  From the groups involved in the direct product, one may construct
  explicitly the generators $\theta_i$, and therefore the degrees
  $(d_i)_i$ appearing in Corollary~\ref{corollary.hilbert_series} in
  the denominator of the Hilbert series.
\end{proof}

\printbibliography

\end{document}